\documentclass[a4paper,10pt]{article}
\usepackage[utf8]{inputenc}
\usepackage[T2A]{fontenc}

\usepackage{amsmath}
\usepackage{amssymb}
\usepackage{amsthm}

\newtheorem{remark}{Remark}
\newtheorem{lemma}{Lemma}
\newtheorem{fact}{Fact}

\newtheorem{theorem}{Theorem}

\newcommand{\diamn}[1]{\langle#1\rangle}
\newcommand{\boxn}[1]{[#1]}

\newcommand{\GLP}{\mathbf{GLP}}
\newcommand{\Olarge}{\mathcal{O}}
\newcommand{\cmp}{\mathbf{c}}
\newcommand{\ordcmp}{\mathbf{oc}}
\newcommand{\ordlog}{\ell}
\newcommand{\cmpscr}[1]{\textsl{\texttt{#1}}}
\newcommand{\Ordlog}{\ell}
\newcommand{\width}{\mathbf{w}}
\newcommand{\codes}{\mathbf{C}}
\newcommand{\ev}{\mathbf{ev}}

\begin{document}
\title{On the complexity of the closed fragment of Japaridze's provability logic}
\author{Fedor Pakhomov\thanks{This work was partially supported by RFFI grant 12-01-00888\_a and Dynasty foundation.}\\Steklov Mathematical Institute,\\Moscow\\ \texttt{pakhfn@mi.ras.ru}}
\date{May 2013}

\maketitle

  \begin{abstract} We consider well-known provability logic $\mathbf{GLP}$. We prove that the $\GLP $-provability problem for variable-free polymodal formulas is $\textsc{PSPACE}$-complete. For a number $n$, let $L^n_0$ denote the class of all polymodal variable-free formulas without modalities $\langle n \rangle,\langle n+1\rangle,\ldots$. We show that, for every number $n$, the $\GLP $-provability problem for formulas from $L^n_0$ is in $\textsc{PTIME}$. 
  \end{abstract}

\section{Introduction} There are some works about computational complexity of provability logics. R.~Ladner in \cite{Lad77} has shown that some logics, including $\mathbf{S4}$, $\mathbf{K}$, and $\mathbf{T}$ have $\textsc{PSPACE}$-complete decision problem. Even though the G\"odel-L\"ob logic $\mathbf{GL}$ was not mentioned in \cite{Lad77} is easy to prove that $\mathbf{GL}$ has a $\textsc{PSPACE}$-complete decision problem. Later it was shown that the $\mathbf{GL}$-provability problem for the formulas with at most one free variable is $\textsc{PSPACE}$-complete \cite{ChagRyb02}\cite{Svej03}. The $\mathbf{GL}$-provability problem for variable-free modal formulas lies in $\textsc{PTIME}$ \cite{ChagRyb02}.

I.~Shapirovsky proved that the decision problem for the Japaridze's logic $\GLP $ lies in $\textsc{PSPACE}$ \cite{Shap08}. Therefore, both the $\GLP $-provability problem for all polymodal formulas and the $\GLP $-provability problem for polymodal formulas with at most one free variable are $\textsc{PSPACE}$-complete.  E.~Dashkov considered the class of all formulas of the form 
$\varphi\longleftrightarrow \psi,$
were $\varphi$ and $\psi$ are built from the logical constant $\top$, conjunction, propositional variables and modalities $\diamn{n}$; he has shown that there exists polynomial time algorithm for the $\GLP $-provability problem for formulas from this class \cite{Dash12}. 

\section{The Logic $\GLP$}

{\it The language of the polymodal provability logic $\GLP$} consists of all formulas well-built of $\top$ (propositional constant for truth), $\bot$ (propositional constant for false), $\land$, $\lor$, $\lnot$, $\to$, $\diamn{0}$, $\diamn{1}$,$\ldots$, $x_0$, $x_1$,$\ldots$ (every natural number can be an index of diamond and an index of variable). We denote this language by $L^\omega_\omega$. Axioms and inference rules of $\GLP$ are
\begin{enumerate}

\setcounter{enumi}{-1}

\item axioms of $\mathbf{PC}$(Propositional Calculus);

\item $\diamn{n} (A\lor B) \to (\diamn{n} A \lor \diamn{n} B)$;

\item $\lnot \diamn{n} \lnot \top$;

\item $\diamn{n} A \to \diamn{n} (A \land \lnot \diamn{n} A)$;

\item $\diamn{n} A \to \diamn{k} A$, for $k\le n$;

\item $\diamn{k} A \to \lnot \diamn{n} \lnot \diamn{k} A$, for $k<n$;

\item $\genfrac{}{}{0.2pt}{0}{A \qquad A\to B}{B}$;

\item $\genfrac{}{}{0.2pt}{0}{A \to B}{\diamn{n} A \to \diamn{n} B}$.
\end{enumerate}

Below we give well-known arithmetical semantics for the logic $\GLP$. We will omit some details of the construction of this semantics; look in \cite{Bek05} for more information. The arithmetical semantics for $\GLP$ were introduced by G.~Japaridze \cite{Jap86} (this semantics is somewhat different from the one we present here). 

Let $L_{\mathrm{FA}}$ be the language of the first-order Peano arithmetic $\mathbf{PA}$, i.e. $L_{\mathrm{FA}}$ is the set of all closed first-order formulas over  the  signature $(=,0,S,+,\cdot)$. All first order theories that we will consider below in the article are theories in this language. $\mathrm{RFN}_n(\mathbf{T})$ is the naturally selected $L_{\mathrm{FA}}$ proposition saying that all $\Sigma_n$ consequences of the recursively axiomatizable theory $\mathbf{T}$ are true. Note that for every recursively axiomatizable theory $\mathbf{T}$, the proposition $\mathrm{RFN}_0(\mathbf{T})$ is the proposition saying that $\mathbf{T}$ is consistent.

 Suppose $\mathbf{T}$ is a recursively axiomatizable first order theory in the signature of $\mathbf{PA}$.  We consider {\it evaluations of $L^\omega_\omega$ formulas} $*\colon L^\omega_\omega\to L_{\mathrm{FA}}$, $\varphi\longmapsto \varphi^{*}$ which enjoys following properties:
\begin{enumerate}
\item $\top^{*}$ is a $\mathbf{T}$-provable proposition, $\bot^{*}$ is a $\mathbf{T}$-disprovable proposition;
\item $(\lnot \varphi)^{*}=\lnot \varphi^{*}$, $(\varphi\land \psi)^{*}=\varphi^{*}\land\psi^{*}$, $(\varphi\lor \psi)^{*}=\varphi^{*}\lor\psi^{*}$, $(\varphi\to \psi)^{*}=\varphi^{*}\to\psi^{*}$, for all $\varphi,\psi\in L^\omega_\omega$;
\item $(\diamn{n}\varphi)^{*}=\mathrm{RFN}_n(\mathbf{T}+\varphi^{*})$, for all $\varphi\in L^\omega_\omega$.
\end{enumerate}

It's known that for omega-correct theories $\mathbf{T}\supset \mathbf{PA}$ the correctness and completeness  theorem for the logic $\GLP$ holds, i.e. for every formula $\varphi\in L^\omega_\omega$
$$\GLP\vdash \varphi \iff \mbox{ for every correct evaluation $*$ we have } \mathbf{T}\vdash \varphi^{*}.$$

Principal applications of the logic $\GLP $ are in proof theory. Consider polymodal formulas of the form 
$$\diamn{n_0}\diamn{n_1}\ldots\diamn{n_{k-1}}\top,$$
where $k\ge 0$ and $n_0,\ldots,n_{k-1}\ge 0$; formulas of this form are known as {\it words}. Those formulas corresponds to arithmetical propositions that are known as {\it iterated reflection principles}
$$\mathrm{RFN}_{n_1}(\mathbf{T}+\mathrm{RFN}_{n_2}(\mathbf{T}+\mathrm{RFN}_{n_3}(\ldots(\mathbf{T}+\mathrm{RFN}_{n_k}(\mathbf{T}))\ldots)).$$ 
This correspondence simplifies the investigation of iterated reflection principles. Iterated reflection principles  were used to obtain a characterization of $\Pi_n$-consequences of $\mathbf{PA}$ and of some fragments of $\mathbf{PA}$. Also iterated reflection principles was used in the proof of the independence of $\mathbf{PA}$ for Beklemishev's Worm Principle \cite{Bek05}. 

For our further purposes we need one simple fact. It is clear that for every word $\alpha$ the theory $\mathbf{PA} +\alpha^{*}$ is omega-correct. Henceforth from G\"odels second incompleteness theorem for omega-correct theories and completeness theorem for $\GLP$ it follows that 
\begin{fact} \label{GLP_not_vdash_fact} Suppose $\alpha$ is a word. Then $$\GLP\not \vdash \alpha\to\diamn{0}\alpha.$$
\end{fact}

All these results mainly exploited  properties of the variable-free fragment of $\GLP $. The variable-free fragment of the logic $\GLP $ is expressive enough to describe a lot of properties of words.

For an ordinal $\alpha\le \omega$ we denote by $L^\alpha_0$ the set of all formulas built from the logical constant $\top$, constant $\bot$, conjunction, disjunction, implication, negation and modalities $\diamn{n}$ for natural numbers $n<\alpha$. For an ordinal $\alpha\le \omega$ we denote by $\GLP^\alpha_0$ the set of all $\GLP$-provable formulas from $L^\alpha_0$. We investigate computational complexity for languages $\GLP ^\alpha_0$.

\section{The fragment $\GLP ^\omega_0$}
\begin{theorem} \label{completeness_theorem} The language $\GLP ^\omega_0$ is $\textsc{PSPACE}$-complete.
\end{theorem}

The language $\GLP ^\omega_0$ lies in $\textsc{PSPACE}$ by Shapirovsky's theorem. Thus, in order to prove Theorem \ref{completeness_theorem}, we need to show that the language $\GLP^\omega_0$ is $\textsc{PSPACE}$-hard. 

By $\mathbf{QBF}$ we denote the language of all true closed quantified Boolean formulas. In order to prove $\textsc{PSPACE}$-hardness of the language $\GLP ^\omega_0$, we construct a polynomial-time reduction from the language $\GLP ^\omega_0$ to the language $\mathbf{QBF}$. 

Suppose we have a formula
$$Q_0x_0Q_1x_1\ldots Q_{n-1} x_{n-1} \varphi(x_0,\ldots,x_{n-1}),$$
where $Q_i\in\{\forall,\exists\}$ and $\varphi(x_0,\ldots,x_{n-1})$ is a Boolean formula with free propositional parameters $x_0,\ldots,x_{n-1}$.  We will construct closed polymodal formulas $\eta_0$ and $\psi_0$ such that $$Q_0x_0Q_1x_1\ldots Q_{n-1} x_{n-1} \varphi(x_0,\ldots,x_{n-1})\mbox{ is true iff }\GLP \vdash \eta_0\;\longleftrightarrow\; \psi_0;$$ 
there are no connective $\longleftrightarrow$ in the language and we express it with the use of $\land$ and $\to$ connectives.

 We construct following formulas:
\begin{itemize}
\item $\eta_n\rightleftharpoons \top$;
\item $\eta_i\rightleftharpoons \diamn{2i}\diamn{4n-2i-1}\top$, for $0\le i<n$;
\item $\theta_i\rightleftharpoons \diamn{2i+1}\diamn{4n-2i-2}\top$, for $0\le i< n$;
\item $\psi_n\rightleftharpoons\varphi[\theta_0,\ldots,\theta_{n-1}/x_0,\ldots,x_{n-1}]$;
\item $\psi_i\rightleftharpoons \diamn{2i}\diamn{4k-2i-1}\diamn{2i}\psi_{i+1}$, for $0\le i<n$ and $Q_i=\exists$.
\item $\psi_i\rightleftharpoons \eta_{i-1}\land \lnot \diamn{2i}\diamn{4k-2i-1}\diamn{2i}(\eta_i\land\lnot \psi_{i+1})$, where $0\le i<n$ and $Q_i=\forall$.
\end{itemize}

For a formula $\xi$ we put $\xi^{\top}\rightleftharpoons\xi$ and $\xi^{\bot}\rightleftharpoons\lnot \xi$.

For every $k\le n$ let $\Lambda_k$($\Lambda_k^{-}$) be the set of all $\sigma\colon \{0,\ldots,k-1\}\to \{\bot,\top\}$ such that 
$$Q_kx_k\ldots Q_{n-1} x_{n-1} \varphi(\sigma(0),\ldots,\sigma(k-1),x_k,\ldots,x_{n-1}) \mbox{ is true(false).}$$ 

The following three lemmas from \cite{Bek05-2} are given here without proof:
\begin{lemma}\cite[Lemma 1]{Bek05-2} \label{ext_lemma1} Suppose $\xi_1$, $\xi_2$ are polymodal formulas and $s_1,s_2$ are natural numbers such that $s_1<s_2$. Then
\begin{enumerate}
\item $\GLP\vdash \diamn{s_2}(\xi_1\land \diamn{s_1}\xi_2) \;\longleftrightarrow\; \diamn{s_2} \xi_1 \land \diamn{s_1}\xi_2$;
\item $\GLP\vdash \diamn{s_2}(\xi_1\land \lnot \diamn{s_1}\xi_2) \;\longleftrightarrow\; \diamn{s_2} \xi_1 \land \lnot\diamn{s_1}\xi_2$.
\end{enumerate}
\end{lemma}

\begin{lemma} \cite[Lemma 10]{Bek05-2}\label{ext_lemma2}Suppose $s$ is a number and $\alpha,\beta$ are words without $\diamn{0},\diamn{1},\ldots,\diamn{s-1}$ such that $\GLP\not\vdash \alpha\to\beta$. Then $\GLP\vdash \diamn{s}(\alpha\land \lnot \beta)\;\longleftrightarrow\;\diamn{s}\alpha$.
\end{lemma}

Next lemma can be proved by straightforward induction on length of $\alpha$ with the use of \cite[Lemma 2]{Bek05-2}
\begin{lemma} \label{ext_lemma3}Suppose $s$ is a number and $\alpha$ is a word without $\diamn{s},\diamn{s+1},\ldots$. Then $$\GLP \vdash \diamn{s}\alpha\;\longleftrightarrow\;\diamn{s}\top.$$
\end{lemma}

\begin{remark} The disjunction  of the empty set of formulas $\bigvee\limits_{\xi\in\emptyset}\xi$ is $\bot$.
The conjunction of the empty set of formulas $\bigwedge\limits_{\xi\in\emptyset}\xi$ is $\top$.
\end{remark} 

\begin{lemma} \label{main_pspace_cmpl_lemma} Suppose $k\le n$. Then $$\GLP \vdash \bigvee\limits_{\sigma\in\Lambda_k}(\eta_{k}\land \bigwedge\limits_{i<k}\theta_i^{\sigma(i)})\;\longleftrightarrow\; \psi_k.$$
\end{lemma}
\begin{proof}
We prove the Lemma  by induction on $n-k$. It is clear that the induction hypothesis holds for $k=n$. Now we prove the inductive step. Consider the case $Q_k=\exists$.
We present the sequence of formulas from $L^\omega_0$ and then prove that neighboring formulas from this sequence are $\GLP$-provably equivalent:
\begin{enumerate}
\item $\psi_k$;
\item $\diamn{2k}\diamn{4n-2k-1}\diamn{2k}\psi_{k+1}$;
\item $\diamn{2k}\diamn{4n-2k-1}\diamn{2k}(\bigvee\limits_{\sigma\in\Lambda_{k+1}}(\eta_{k+1}\land \bigwedge\limits_{i<k+1}\theta_i^{\sigma(i)}))$;
\item $\bigvee\limits_{\sigma\in\Lambda_{k+1}}\diamn{2k}\diamn{4n-2k-1}\diamn{2k}(\eta_{k+1}\land \bigwedge\limits_{i<k+1}\theta_i^{\sigma(i)}))$;
\item $\bigvee\limits_{\sigma\in\Lambda_{k+1}}(\diamn{2k}\diamn{4n-2k-1}\diamn{2k}(\eta_{k+1}\land \theta_k^{\sigma(k)}) \land \bigwedge\limits_{i<k}\theta_i^{\sigma(i)})$;
\item $\bigvee\limits_{\sigma\in\Lambda_{k+1}}(\diamn{2k}\diamn{4n-2k-1}\top \land \bigwedge\limits_{i<k}\theta_i^{\sigma(i)})$ ;
\item  $\bigvee\limits_{\sigma\in\Lambda_{k}}(\eta_k \land \bigwedge\limits_{i<k}\theta_i^{\sigma(i)})$.
\end{enumerate}
 Clearly, that pairs of formulas $\langle 1.,2.\rangle$, $\langle 2.,3. \rangle$, $\langle 3.,4. \rangle$, and $\langle 6.,7.\rangle$ are pairs of $\GLP$-provable equivalent formulas. The equivalence between 4. and 5. can be obtained by iterative application of Lemma \ref{ext_lemma1}.  In order to prove the $\GLP$-equivalence between 5. and 6. we prove that

\begin{equation}\label{3247_form1}\GLP\vdash\diamn{4n-2k-1}\diamn{2k}(\eta_{k+1}\land \theta_k)\;\longleftrightarrow\; \diamn{4n-2k-1}\top,\end{equation}
\begin{equation}\label{3247_form2}\GLP\vdash\diamn{4n-2k-1}\diamn{2k}(\eta_{k+1}\land \lnot \theta_k)\;\longleftrightarrow\; \diamn{4n-2k-1}\top.\end{equation}

We have $$\GLP \vdash \eta_{k+1}\land \theta_k\;\longleftrightarrow\; \diamn{2k+2}\diamn{4n-2k-3}\diamn{2k+1}\diamn{4n-2k-2}\top$$ by Lemma \ref{ext_lemma1}. From Lemma \ref{ext_lemma3} it follows that $$\GLP\vdash \diamn{4n-2k-1}\diamn{2k}\diamn{2k+2}\diamn{4n-2k-3}\diamn{2k+1}\diamn{4n-2k-2}\top\;\longleftrightarrow\;\diamn{4n-2k-1}\top$$. Thus equivalence (\ref{3247_form1}) holds.

 Let us prove that $\GLP\not\vdash\eta_{k+1}\to \theta_{k}$. Assume converse, $\GLP\vdash\eta_{k+1}\to \theta_{k}.$ Then we have
$$
\begin{aligned}
\GLP\vdash \diamn{0}\diamn{2k+2}\diamn{4n-2k-3}\top& \to \diamn{0}\diamn{2k+1}\diamn{4n-2k-2}\top\\
& \to \diamn{0}\diamn{0}\diamn{4n-2k-2}\top\\
& \to \diamn{0}\diamn{0}\diamn{4n-2k-2}\diamn{4n-2k-3}\top\\
&\to \diamn{0}\diamn{0}\diamn{2k+2}\diamn{4n-2k-3}\top.
\end{aligned}
$$
But by Fact \ref{GLP_not_vdash_fact} we have $$\GLP\not \vdash \diamn{0}\diamn{2k+2}\diamn{4n-2k-3}\top \to \diamn{0}\diamn{0}\diamn{2k+2}\diamn{4n-2k-3}\top.$$ Contradiction. 
Therefore $\GLP\not\vdash\eta_{k+1}\to \theta_{k}$. 

By Lemma \ref{ext_lemma2} we have $$\GLP\vdash\diamn{2k}(\eta_{k+1}\land \lnot \theta_k)\;\longleftrightarrow \diamn{2k}\diamn{2k+2}\diamn{2n-2k-3}\top.$$ Hence by Lemma \ref{ext_lemma3} the equivalence (\ref{3247_form2}) holds. 

Therefore formulas 5. and 6. are $\GLP$-provably equivalent. Finally, we conclude that formulas 1. and 7. are $\GLP$-provable equivalent. This finish the proof of the inductive step in the case $Q_k=\exists$.

 Now we switch to the case $Q_k=\forall$. We consider the following sequence of formulas from $L^\omega_0$:
\begin{enumerate} 
\item $\psi_k$;
\item $\eta_k\land \lnot \diamn{2k}\diamn{4n-2k-1}\diamn{2k}(\eta_{k+1}\land \lnot \psi_{k+1})$;
\item $\eta_k\land \lnot \diamn{2k}\diamn{4n-2k-1}\diamn{2k}(\eta_{k+1}\land \lnot( \bigvee\limits_{\sigma\in\Lambda_{k+1}}(\eta_{k+1}\land \bigwedge\limits_{i<k+1}\theta_i^{\sigma(i)})))$;
\item $\eta_k\land \lnot \diamn{2k}\diamn{4n-2k-1}\diamn{2k}(\bigvee\limits_{\sigma\in\Lambda^{-}_{k+1}}(\eta_{k+1}\land \bigwedge\limits_{i<k+1}\theta_i^{\sigma(i)}))$;
\item $\eta_k\land \lnot (\bigvee\limits_{\sigma\in\Lambda^{-}_{k}}(\eta_{k}\land \bigwedge\limits_{i<k}\theta_i^{\sigma(i)}))$;
\item  $\bigvee\limits_{\sigma\in\Lambda_{k}}(\eta_k \land \bigwedge\limits_{i<k}\theta_i^{\sigma(i)})$.
\end{enumerate}
All equivalences between neighboring formulas in last sequence but the equivalence between 4. and 5. holds obviously. The last equivalence can be proved in the same way as equivalency between formulas 3. and 7. from the proof of the inductive step for the case $Q_k=\exists$.
\end{proof} 

By Lemma \ref{main_pspace_cmpl_lemma} we have $\GLP\vdash \eta_0\;\longleftrightarrow \psi_0$ if the formula $$Q_0x_0Q_1x_1\ldots Q_{n-1} x_{n-1} \varphi(x_0,x_1,\ldots,x_{n-1})$$ is true and $\GLP\vdash \bot\;\longleftrightarrow \psi_0$ if that formula is false. Using the arithmetic semantics for $\mathbf{GLP}$ we easily obtain $\GLP\not\vdash \eta_0\;\longleftrightarrow \bot$. Hence $$Q_0x_0Q_1x_1\ldots Q_{n-1} x_{n-1} \varphi(x_0,\ldots,x_{n-1})\mbox{ is true iff }\GLP \vdash \eta_0\;\longleftrightarrow\; \psi_0.$$

It is easy to check that the formula $\eta_0\;\longleftrightarrow\;\psi_0$ is constructed in polynomial time in length of $Q_0x_0Q_1x_1\ldots Q_{n-1} x_{n-1} \varphi(x_0,x_1,\ldots,x_{n-1})$. This gives us the reduction and finish the proof of Theorem \ref{completeness_theorem}.

\section{Fragments $\GLP ^n_0$}
The method we describe in the previous section essentially use an infinite number of modalities. Therefore for every finite $n$ this method cannot be used to prove $\textsc{PSPACE}$-hardness of $\GLP ^n_0$.

We prove in this section
 \begin{theorem} \label{n_fragment} For every number $n$, the language $\GLP ^n_0$ lies in $\textsc{PTIME}$.
\end{theorem}

First, we give a plan of our proof. We use a Kripke model $\mathcal{U}^n_{\omega_n}$ such that the fragment $\GLP^n_0$ is complete with respect to this model. For every formula from $L^n_0$ there is a corresponding set of $\mathcal{U}^n_{\omega_n}$-worlds (the set of all worlds that satisfy this formula). Completeness of $\GLP^n_0$ with respect to $\mathcal{U}^n_{\omega_n}$ means that a formula from $L^n_0$ lies in $\GLP^n_0$ iff every world of $\mathcal{U}^n_{\omega_n}$ lies in the corresponding set. Of course for a formula $\varphi\in L^n_0$ the corresponding set can be obtain by interpreting propositional constants in $\varphi$ and propositional connectives in $\varphi$ as special sets of $\mathcal{U}^n_{\omega_n}$-worlds and special operations on sets of $\mathcal{U}^n_{\omega_n}$-worlds, respectively.  
 We use special codes to encode sets of $\mathcal{U}^n_{\omega_n}$-worlds (note that there exist sets without a corresponding code). In a decision algorithm we use computable functions $\cmpscr{Intr}(x,y)$, $\cmpscr{Cmpl}(x)$, $\cmpscr{RInv}_0(x),\ldots,\cmpscr{RInv}_{n-1}(x)$, $\cmpscr{IsEmp}(x)$ to manipulate codes (in the complete proof below these functions have additional arguments and parameters). Our decision algorithm  works this way:
\begin{enumerate}
\item We get an input formula $\varphi\in L^n_0$.
\item We switch to an formula $\varphi'$ such that $$\GLP\vdash \varphi'\longleftrightarrow \varphi$$ and $\varphi'$ is build of $\bot,\land,\lnot,
\diamn{0},\ldots,\diamn{n-1}$. We construct $\varphi'$ by straightforward translation.
\item We build a code $c(\varphi')$ for the set that corresponds to $\varphi'$. In order to do that we define the mapping $c$ of $\varphi'$ subformulas to codes. Function $c$ is given by the following rules:
\begin{enumerate}
\item $c(\bot)$ is the constant code for empty set;
\item $c(\psi_1\land\psi_2)$ is $\cmpscr{Intr}(c(\psi_1),c(\psi_2))$;
\item $c(\lnot\psi)$ is $\cmpscr{Cmpl}(c(\psi))$;
\item $c(\diamn{k}\psi)$ is $\cmpscr{RInv}_k(c(\psi))$.
\end{enumerate}
\item We accept $\varphi$ iff $\cmpscr{IsEmp}$ returns positive answer on input $\cmpscr{Cmpl}(c(\varphi'))$.
\end{enumerate}
 Further we describe the way we estimate the algorithm running time. We introduce functions $\cmp^n_{\omega_n}$ and $\width^n_{\omega_n}$ to measure complexity of codes. We prove bounds on the complexity of resulting codes and running time for functions $\cmpscr{Cmpl},\cmpscr{Intr},\cmpscr{RInv}_{0},\cmpscr{RInv}_{1},\ldots,\cmpscr{RInv}_{n-1}$  in the terms of complexity of input codes. This gives us the estimation for running time of our decision algorithm. In most of the lemmas below we simultaneously construct a computable function with desired properties and prove bounds for this function.

Now we are going to give a precise definition of Kripke models we use. The definition of models $\mathcal{U}^n_\alpha$ uses the notion of ordinal number. In this section we denote ordinal numbers by lower case Greek letters $\alpha,\beta,\gamma,\delta,\zeta$; we denote by $\mathbf{On}$ the class of all ordinals.

\begin{fact}[Cantor Normal Form Theorem]\label{CNF_theorem} Every ordinal $\alpha$ can be presented in a unique way as a sum
$$\alpha=\omega^{\beta_0}+\ldots+\omega^{\beta_{n-1}}$$
such that $\beta_0\ge\beta_1\ge\ldots\ge\beta_{n-1}$ and $n\ge 0$.
\end{fact}
 Let {\it the function $\ordlog\colon\mathbf{On}\to\mathbf{On}$} be given by 
\begin{itemize}
\item $\ordlog(0)=0$;
\item $\ordlog(\alpha)=\beta_{n-1}$, where $\alpha>0$ and Cantor normal form of $\alpha$ is $\omega^{\beta_0}+\ldots+\omega^{\beta_{n-1}}$.
\end{itemize}

We use following notations for ordinals:
\begin{itemize}
\item $\omega_0=1$;
\item $\omega_{n+1}=\omega^{\omega_n}$;
\item $\varepsilon_0=\sup\limits_{n\to \omega}\omega_n$.
\end{itemize}
Ordinal $\varepsilon_0$ is the first ordinal $\alpha$ such that $\omega^\alpha=\alpha$. 

We present a definition of {\it Ignatiev's Model} $\mathcal{U}=(U,R_0,R_1,\ldots)$ \cite{Ign93}. The set $U$ is the set of all sequences 
$$(\alpha_0,\alpha_1,\alpha_2,\ldots)$$
such that every $\alpha_i$ is an ordinal, $\alpha_0<\varepsilon_0$ and $\alpha_{i+1}\le \ordlog(\alpha_i)$, for every $i\in\omega$.  For every number $k$ the binary relation $R_k$ is given by
$$(\alpha_0,\alpha_1,\ldots)R_k(\beta_0,\beta_1,\ldots)\stackrel{\mathrm{def}}{\iff} \beta_k<\alpha_k \& \forall i<k( \alpha_i=\beta_i).$$
The model $\mathcal{U}$ is the universal model for the closed fragment of $\GLP $ \cite{Ign93}\cite{BekJooVer05}. For every formula $\varphi\in L^\omega_0$ we have
$$\GLP \vdash\varphi \iff  \varphi\mbox{ is valid in $\mathcal{U}$}.$$
It is easy to see that for every sequence $(\alpha_0,\alpha_1,\ldots)\in U$ we have $\alpha_i=0$, for enough big $i$.

Actually we will work with ``smaller'' {\it models $\mathcal{U}_\alpha^n =(U^n_\alpha,R_0,R_1,\ldots,R_{n-1})$} for $1\le\alpha<\varepsilon_0$ and $n\ge 0$. For $\alpha<\varepsilon_0$ the set $U^n_\alpha\subset U$ is the set of all sequences of ordinals
$$(\alpha_0,\alpha_1,\ldots,\alpha_{n-1})$$
such that $\alpha_0<\alpha$ and $\alpha_{i+1}\le \ordlog(\alpha_i)$ for every $i< n-1$. For every $k<n$ the binary relation $R_k$ is given by
$$(\alpha_0,\alpha_1,\ldots,\alpha_{n-1})R_k(\beta_0,\beta_1,\ldots,\beta_{n-1})\stackrel{\mathrm{def}}{\iff} \beta_k<\alpha_k \& \forall i<k( \alpha_i=\beta_i).$$ 
We note several properties of models $\mathcal{U}^n_\alpha$
\begin{fact} \label{U_n_alpha_fact} Suppose $n$ is a number and $\alpha$ is an ordinals, $0<\alpha<\varepsilon_0$.
\begin{enumerate}
\item the only element of $U^0_\alpha$ is $()$;
\item for every $\alpha_0$ the model $(\{(\beta_0,\beta_1,\ldots,\beta_{n-1})\in U^n_\alpha\mid \alpha_0=\beta_0\},R_1,\ldots,R_{n-1})$ is isomorphic to the model $\mathcal{U}^{n-1}_{\ordlog(\alpha_0)+1}$;
\item for every $k$ from $1$ to $n$ and $(\beta_0,\ldots,\beta_{n-1}),(\gamma_0,\ldots,\gamma_{n-1})\in U^n_\alpha$ such that $$(\beta_0,\ldots,\beta_{n-1})R_k(\gamma_0,\ldots,\gamma_{n-1}),$$ we have $\beta_0=\gamma_0$.
\end{enumerate}
\end{fact}

The model $\mathcal{U}^n_{\omega_n}$ is the universal model for the closed $n$-modal fragment of $\GLP $ \cite{Ign93}\cite{BekJooVer05}. For every formula $\varphi\in L^n_0$ we have
$$\GLP \vdash\varphi \iff  \varphi\mbox{ is valid in $\mathcal{U}^n_{\omega_n}$}.$$

We will use the well-known $\textsc{RAM}$(random access machine) calculation model. More specifically, we will use the variant of $\textsc{RAM}$ from \cite{CookRec73} with the execution time for every instruction equal to $1$. All time bounds in present paper are given for this model. In \cite{CookRec73} it was shown that $\textsc{RAM}$ can be simulated on a multi-tape Turing machine with at most cubic running time growth.
 
We will effectively encode some subsets of $U^n_\alpha$. In order to do it we will use the following encoding of ordinals less then $\varepsilon_0$ known as Cantor ordinal notations. We encode expressions in Cantor normal forms
$$\omega^{\beta_0}+\ldots+\omega^{\beta_{n-1}},$$
where every $\beta_i$ is also encoded this way. Obviously, this gives us unique (by Fact \ref{CNF_theorem}) encoding for every ordinal less than $\varepsilon_0$.

All ordinals that we use below are less than $\varepsilon_0$. Below, we consider only ordinals that are less than $\varepsilon_0$. We don't make a distinguish between an ordinal $<\varepsilon_0$ and it's encoding.

We define {\it the function $\cmp\colon \varepsilon_0\to\omega$}. For an ordinal $\alpha=\omega^{\beta_0}+\ldots+\omega^{\beta_{n-1}}$ in Cantor normal form we put
$$\cmp(\alpha)=1+\cmp(\beta_0)+\ldots+\cmp(\beta_{n-1}).$$
Obviously, this gives us a unique function $\cmp$. For an ordinal $\alpha$, the amount of memory which is needed to store the code of $\alpha$ is $\Olarge(\cmp(\alpha))$.

We omit the proofs of two following the lemmas:
\begin{lemma} Ordinals $\alpha,\beta$ can be compared within time $\Olarge(\cmp(\alpha)+\cmp(\beta))$.
\end{lemma} 
\begin{proof} We will describe recursive algorithm. Suppose $\alpha=\omega^{\alpha_0}+\ldots+\omega^{\alpha_{n-1}}$ in Cantor normal form and $\beta=\omega^{\beta_0}+\ldots+\omega^{\beta_{m-1}}$ in Cantor normal form. Starting from $i=0$ we increase $i$ by $1$ until $i\ge\min(n,m)$ or $\alpha_i\ne\beta_i$, here we use recursive calls to compare $\alpha_i$ and $\beta_i$. If after this procedure $i=n=m$ then $\alpha=\beta$. If $i=n<m$ or $i<\min(n,m)$ and $\alpha_i<\beta_i$ then $\alpha<\beta$. Otherwise, $\alpha>\beta$.
 
The required time bound for this algorithm obviously holds.
\end{proof}

\begin{lemma} For ordinals $\alpha,\beta$, we can find an ordinal $\alpha+\beta$ within time $\Olarge(\cmp(\alpha)+\cmp(\beta))$ and we have $\cmp(\alpha+\beta)\le\cmp(\alpha)+\cmp(\beta)$.
\end{lemma}
\begin{proof} The cases of $\beta=0$ or $\alpha=0$ are trivial. Below we assume that $\alpha,\beta>0$. Suppose $\alpha=\omega^{\alpha_0}+\ldots+\omega^{\alpha_{n-1}}$ in Cantor normal form and $\beta=\omega^{\beta_0}+\ldots+\omega^{\beta_{m-1}}$ in Cantor normal form. We find the smallest $k<n$ such that $\alpha_i<\beta_0$. Obviously, the Cantor normal form of $\alpha+\beta$ is
$$\omega^{\alpha_0}+\ldots+\omega^{\alpha_{k-1}}+\omega^{\beta_k}+\ldots+\omega^{\beta_{m-1}}.$$

Linear time bound for this algorithm obviously holds.
\end{proof}


For ordinals $\alpha$ and $\beta$, $\alpha<\beta$ we encode the interval $[\alpha,\beta)=\{\gamma\mid \alpha\le\gamma <\beta\}$ by the pair $\langle \alpha,\beta \rangle$. Only intervals we consider in the present paper are intervals of this form. For an interval $A=[\alpha,\beta)$ we put $$\Ordlog(A)=\sup\{\ordlog(\gamma)+1\mid \gamma\in A\}.$$

\begin{lemma} \label{Ordlog_lemma} For a given interval  $A=[\alpha,\beta)$ such that $\ordlog(\alpha)=0$
\begin{enumerate}
\item  we can find $\Ordlog(A)$ within time $\Olarge(\cmp(\alpha)+\cmp(\beta))$;
\item  $\cmp(\Ordlog(A))\le \cmp(\beta)$;
\item \label{Ordlog_lemma_item3} $[0,\Ordlog(A))=\{\ordlog(\gamma)\mid\gamma\in A\}$.
\end{enumerate}
\end{lemma}
\begin{proof}
Suppose Cantor normal forms of ordinals $\alpha$ and $\beta$ are
$$\alpha=\omega^{\alpha_0}+\ldots+\omega^{\alpha_{n-1}}\mbox{ and}$$
$$\beta=\omega^{\beta_0}+\ldots+\omega^{\beta_{m-1}}.$$
 Let $k=\min(\{n\}\cup \{i\mid \alpha_i<\beta_i\})$. Obviously, $k<m$. 

Let $\zeta=\max(\beta_k,1)$ if $k=m$, and let $\zeta=\beta_k+1$ otherwise. 

\textbf{Claim:} $\{\ordlog(\gamma)\mid\gamma\in A)\}=[0,\zeta)$. 

First, we consider any $\gamma\in A$ and prove that $\ordlog(\gamma)<\zeta$.  The Cantor normal form of the ordinal $\gamma$ is
$$\gamma=\omega^{\alpha_0}+\ldots+\omega^{\alpha_{k-1}}+\omega^{\gamma_0}+\ldots+\omega^{\gamma_{l-1}},$$
where all $\gamma_i\le \beta_k$ and if $k=m$ then all $\gamma_i<\beta_k$. If $l=0$ then we have $\alpha=\gamma$ and $\ordlog(\gamma)=\ordlog(\alpha)=0$. 

Suppose $l\ne 0$. We have $\ordlog(\gamma)=\gamma_{l-1}\le \beta_k$ if $k>m$ and we have  $\ordlog(\delta)=\gamma_{l-1}< \beta_k$ if $k=m$.  Hence  $\ordlog(\gamma)<\zeta$. 

Now, we consider any $\delta\in [0,\zeta)$ and find a $\gamma\in A$ such that $\ordlog(\gamma)=\delta$. If $\delta=0$ then we can choose $\gamma=\alpha$. Otherwise we choose $\gamma=\alpha+\omega^{\delta}$; obviously then $\gamma\in A$.  This complete the proof of the claim. 

Therefore $\Ordlog(A)=\zeta$. Obviously, $\zeta$ can be found within time $\Olarge(\cmp(\alpha)+\cmp(\beta))$. We have $\cmp(\zeta)\le \cmp(\beta_k)+1\le \cmp(\beta)$. This finishes the proof of the Lemma.
\end{proof}

For every ordinal $\alpha>0$ and every number $n$ we will define the set of codes $\codes^n_\alpha$ and the evaluation function $\ev^n_\alpha\colon \codes^n_\alpha\to \mathcal{P}(U^n_\alpha)$. Every set $\codes^0_\alpha$ is just the set $\{0,1\}$. We put $\ev^0_\alpha(0)=\emptyset$ and $\ev^0_\alpha(1)=\{()\}$ (the set contains only the empty sequence $()$). For $n>0$ elements of $\codes^n_\alpha$ are tuples consists of
\begin{enumerate}
\item number $m$;
\item $A_0,\ldots,A_{m-1}$, where every $A_i$ is a nonempty interval $[\beta_i,\gamma_i)$, for every $i$, we have $\beta_i\ne \in \mathrm{Lim}$, $\bigsqcup\limits_{i<m} A_i=[0,\alpha)$, and for all $i<m-1$, we have $\gamma_i=\beta_{i+1}$;
\item $c_0\in U^{n-1}_{\Ordlog(A_0)},\ldots,c_{m-1}\in U^{n-1}_{\Ordlog(A_{n-1})}$.
\end{enumerate}
From formal point of view for $n>0$ an element $d\in \codes^n_\alpha$ is a triple $(n,\overline{A},\overline{c})$. We put $$\ev^n_\alpha((m,\overline{A},\overline{c}))=\bigsqcup\limits_{i<m} \{(\beta_0,\beta_1,\ldots,\beta_{n-1})\in U^n_\alpha\mid \beta_0\in A_i, (\beta_1,\ldots,\beta_{n-1})\in \ev^{n-1}_{\ordlog(A_i)}(c_i)\}.$$
Suppose $A$ is a subset of $U^n_\alpha$ and $c$ is an element of $\codes^n_\alpha$ such that $\ev^n_\alpha(c)=A$. Then we say that $c$ is a code for $A$. 

For every number $n$ and every ordinal $\alpha\in [1,\varepsilon_0)$, we introduce functions $\width^n_\alpha\colon \codes^n_\alpha\to \omega$ and $\ordcmp^n_\alpha\colon \codes^n_\alpha\to \omega$ in order to measure complexity of codes
\begin{itemize}
\item $\width^0_\alpha(c)=1$;
\item $\width^{n+1}_\alpha((m,\overline{A},\overline{c}))=\max(\{m\}\cup \{\width^n_{\alpha}(c_i)\mid i<m\})$;
\item $\ordcmp^0_\alpha(c)=1$;
\item $\ordcmp^{n+1}_\alpha((m,\overline{A},\overline{c}))=\max\limits_{i<n}\cmp(\beta_i)\;+\;\max\limits_{i<n}\ordcmp^n_{\Ordlog([\beta_i,\gamma_i))}(c_i)$, where every $A_i=[\beta_i,\gamma_i)$.
\end{itemize}

Obviously, the following two lemmas holds
\begin{lemma} \label{cmpest_lemma}Suppose $c=(m,\overline{A},\overline{d})\in\codes^n_\alpha$ and for every $i<m$, the interval $A_i=[\beta_i,\gamma_i)$. Then for all $i<m$, we have $\cmp(\beta_i)\le\ordcmp(c)$, $\cmp(\gamma_i)\le \max(\cmp(\alpha),\ordcmp(c))$, $\ordcmp(d_i)\le \ordcmp(c)$, and $\cmp(\Ordlog(A_i))\le \max(\cmp(\alpha),\ordcmp(c))$.
\end{lemma}

\begin{lemma} The amount of memory which is needed to store a given code $c\in \codes^n_\alpha$ is $\Olarge((\ordcmp^n_{\alpha}(c)+\cmp(\alpha))\cdot(\width^n_\alpha(c))^n)$.
\end{lemma}

\begin{lemma} Suppose $n\ge 0$. Then there exists a computable function $\cmpscr{IsEmp}_n(\alpha,c)$ such that for arguments $0<\alpha<\varepsilon_0$ and $c\in \codes^n_\alpha$:
\begin{enumerate}
\item $\cmpscr{IsEmp}_n$ returns $1$ if $\ev^n_\alpha(c)=\emptyset$ and returns $0$ otherwise;
\item running time of $\cmpscr{IsEmp}_n$ is $\Olarge(\max(\ordcmp^n_\alpha(c),\cmp(\alpha))\cdot(\width^n_\alpha(c))^n)$.
\end{enumerate}
\end{lemma}
\begin{proof}We prove this Lemma by induction on $n$. Suppose $n=0$. Then for a given $\alpha$ and $c$, we have $\ev^n_\alpha(c)=\emptyset$ iff $c=0$. This gives us the function $\cmpscr{IsEmp}_0$.

Now we consider the case of $n>0$.  From the Lemma \ref{Ordlog_lemma} Item \ref{Ordlog_lemma_item3} it follows that for a given  ordinal $\alpha$ and a given code $c=(k,\overline{A},\overline{d})$ the evaluation $\ev^n_\alpha(c)=\emptyset$ iff for all $i<k$ we have $\ev^{n-1}_{\Ordlog(A_i)}(d_i)=\emptyset$. Whether the right part of the last equivalence holds can be checked by $k$ calls of $\cmpscr{IsEmp}_{n-1}$ with $\Olarge(\max(\ordcmp^n_\alpha(c),\cmp(\alpha))\cdot(\width^n_\alpha(c))^{n-1})$ time upper bound each. This gives us $\cmpscr{IsEmp}_n$ with running time $\Olarge(\max(\ordcmp^n_\alpha(c),\cmp(\alpha))\cdot(\width^n_\alpha(c))^n)$.
\end{proof}

\begin{lemma} Suppose $n\ge 0$. Then there exists a computable function $\cmpscr{Cmpl}_n(\alpha,c)$ such that for arguments $0<\alpha<\varepsilon_0$ and  $c\in\codes^n_\alpha$:
\begin{enumerate}
\item $\cmpscr{Cmpl}_n(\alpha,c)\in\codes^n_\alpha$;
\item $\ev^n_\alpha(\cmpscr{Cmpl}_n(\alpha,c))=U^n_\alpha\setminus \ev^n_\alpha(c)$;
\item $\ordcmp^n_\alpha(\cmpscr{Cmpl}_n(\alpha,c))=\ordcmp^n_\alpha(c)$,  $\width^n_\alpha(\cmpscr{Cmpl}_n(\alpha,c))=\width^n_\alpha(d)$;
\item running time of $\cmpscr{Cmpl}_n$ is $\Olarge(\max(\ordcmp^n_\alpha(c),\cmp(\alpha))\cdot(\width^n_\alpha(c))^n)$.
\end{enumerate}
\end{lemma}
\begin{proof} We prove this Lemma by induction on $n$. 

Consider the case $n=0$. Suppose an input $(\alpha,c)$ is given. Then we put $\cmpscr{Cmpl}_0(\alpha,c)=0$ if $c=1$ and we put $\cmpscr{Cmpl}_0(\alpha,c)=1$ otherwise. Obviously, this gives us $\cmpscr{Cmpl}_0$ that satisfies all required conditions.

Now, consider the case of $n>0$. Suppose an input $(\alpha,c)$ is given and $c=(m,\overline{A},\overline{d})$. We put $\cmpscr{Cmpl}_n(\alpha,c)=(m,\overline{A},\overline{e})$, where the vector $\overline{e}=(e_0,\ldots,e_{m-1})$ and for every $i<m$, $e_i=\cmpscr{Cmpl}_{n-1}(\Ordlog(A_i),d_i)$. This gives us the computable function $\cmpscr{Cmpl}_n$. Obviously, $\ordcmp^n_\alpha(\cmpscr{Cmpl}_{n}(\alpha,c))=\ordcmp^n_\alpha(c)$,  $\width^n_\alpha(\cmpscr{Cmpl}_{n}(\alpha,c))=\width^n_\alpha(d)$. From Lemma \ref{Ordlog_lemma} and inductive hypothesis it follows that $\ev^n_\alpha(\cmpscr{Cmpl}_n(\alpha,c))=U^n_\alpha\setminus \ev^n_\alpha(c)$.  From inductive hypothesis and Lemma \ref{cmpest_lemma} it follows that running time of $\cmpscr{Cmpl}_n(\alpha,c)$ is $\Olarge(\width^n_\alpha(c)^n\cdot\max(\cmp(\alpha),\ordcmp(c)))$. 
\end{proof}

Obviously, the following lemma holds
\begin{lemma} \label{Emptyset_lemma}Suppose $n\ge 0$. Then there exists a computable function $\cmpscr{EmpS}_n(\alpha)$ such that for a given ordinal $\alpha>0$ it returns within time $\Olarge(\cmp(\alpha))$ a code $c\in \codes^n_\alpha$ such that $\ev^n_\alpha(c)=\emptyset$, $\ordcmp^n_\alpha(c)=n$, and $\width^n_\alpha(c)=1$ .\end{lemma}

\begin{lemma} \label{Infinum_lemma}Suppose $n>0$. Then there exists a computable function $\cmpscr{Inf}_n(\alpha,c)$ such that for arguments $0<\alpha<\varepsilon_0$ and $c\in\codes^n_\alpha$, $\ev^n_\alpha(c)\ne \emptyset$:
\begin{enumerate}
\item $\cmpscr{Inf}_n(\alpha,c)$ is an ordinal;
\item \label{infinum_lemma_item3}$\cmpscr{Inf}_n(\alpha,c)=\inf\{\gamma_0\mid \exists \gamma_1,\ldots,\gamma_{n-1}((\gamma_0,\ldots,\gamma_{n-1})\in \ev^n_\alpha(c))\}$;
\item $\cmp(\cmpscr{Inf}_n(\alpha,c))\le\ordcmp(c)$; 
\item \label{infinum_lemma_item5}$\cmpscr{Inf}_n$ running time is $\Olarge(\max(\ordcmp^n_\alpha(c),\cmp(\alpha))\cdot(\width^n_\alpha(c))^n)$.
\end{enumerate}
\end{lemma}
\begin{proof}
 We prove the Lemma by induction on $n$. Suppose an input $(\alpha,c)$ is given, $c=(m,\overline{A},\overline{d})\in\codes^n_\alpha$, and for every $i<m$, $A_i=[\beta_i,\gamma_i)$. We choose the minimal $k$ such that $\cmpscr{IsEmp}_n(d_k)=0$ (recall that it mean that $\ev^n_\alpha(d_k)\ne \emptyset$); such a $k$ exists because of $\ev^n_\alpha(c)\ne\emptyset$. In the case of $n=1$ we return $\beta_k$ as the result of $\cmpscr{Inf}_n(\alpha,c)$. From the Lemma \ref{cmpest_lemma} it follows that $\cmp(\cmpscr{Inf}_n(\alpha,c))\le\ordcmp^n_\alpha(c)$. The required algorithm running time upper bound obviously holds. Now we consider the case of $n>1$. Let $\delta=\cmpscr{Inf}_{n-1}(\Ordlog(A_k),d_k)$. If $\delta=0$ then we put $\cmpscr{Inf}_n(\alpha,c)=\beta_k$. Otherwise we put $\cmpscr{Inf}_n(\alpha,c)=\beta_k+\omega^\delta$. From the inductive hypothesis it follows that \ref{infinum_lemma_item3}. and \ref{infinum_lemma_item5}. holds. By inductive hypothesis we have $$
\begin{aligned}\cmp(\cmpscr{Inf}_n(\alpha,c))&=\cmp(\beta_k)+\cmp(\cmpscr{Inf}_{n-1}(\Ordlog(A_k),d_k))\le \cmp(\beta_k)+\ordcmp(d_k)\\&\le \max\limits_{i<m}\cmp(\beta_i)+\max\limits_{i<m}\ordcmp(d_i)=\ordcmp(c)\end{aligned}$$  
\end{proof}

\begin{lemma} Suppose $n> 0$ and $k<n$. Then there exists a computable function $\cmpscr{RInv}_{n,k}(\alpha,c)$ such that for arguments $0<\alpha<\varepsilon_0$ and $c\in \codes^n_\alpha$:
\begin{enumerate}
\item $\cmpscr{RInv}_{n,k}(\alpha,c)$ is an element of $\codes^n_\alpha$;
\item $\ev^n_\alpha(\cmpscr{RInv}_{n,k}(\alpha,c))=\{w\in U^n_\alpha\mid \exists w'\in \ev^n_\alpha(c) (w R_n w')\}$;
\item $\ordcmp^n_\alpha(\cmpscr{RInv}_{n,k}(\alpha,c))\le \ordcmp^n_\alpha(c)+n$, $\width^n_\alpha(\cmpscr{RInv}_{n,k}(\alpha,c))\le \width^n_\alpha(c)+1$;
\item $\cmpscr{RInv}_{n,k}(\alpha,c)$ running time is $\Olarge(\max(\ordcmp^n_\alpha(c),\cmp(\alpha))\cdot(\width^n_\alpha(c))^n)$. 
\end{enumerate}
\end{lemma}
\begin{proof} We will prove this Lemma by induction on $k$. 

We consider the case of $k=0$. Suppose an ordinal $\alpha$ and a code $c\in\codes^n_\alpha$ are given. If $\cmpscr{IsEmp}(\alpha,c)=1$, i.e. $\ev^n_\alpha(c)$ is empty, then we put $$\cmpscr{RInv}_{n,k}(\alpha,c)=\cmpscr{EmpS}_n(\alpha).$$ Otherwise, we put $$\cmpscr{RInv}_{n,k}(\alpha,c)=(2,(B_0,B_1),(e_0,e_1)),$$ where $B_0=[0,\cmpscr{Inf}_n(\alpha,c)+1)$, $B_1=[\cmpscr{Inf}_n(\alpha,c)+1,\alpha)$, $e_0=\cmpscr{EmpS}_{n-1}(\Ordlog(B_0))$, and $e_1=\cmpscr{Cmpl}_n(\Ordlog(B_1),\cmpscr{EmpS}_{n-1}(\Ordlog(B_1)))$.

Now we consider the case of $k>0$. Suppose  an ordinal $\alpha>0$ and a code $c=(m,\overline{A},\overline{d})\in\codes^n_\alpha$ are given. For $i<m$ let $e_i=\cmpscr{RInv}_{n-1,k-1}(\Ordlog(A_i),d_i)$. We put $\cmpscr{RInv}_{n,k}(\alpha,c)=(m,\overline{A},\overline{e})$. All required conditions on function $\cmpscr{RInv}_{n,k}$ can be checked easily; inductive hypothesis, Fact \ref{U_n_alpha_fact}, Lemma \ref{Infinum_lemma}, and Lemma \ref{Emptyset_lemma} are used here.  
\end{proof}

\begin{lemma} Suppose $n\ge 0$. Then there exists a computable function $\cmpscr{Rstr}_{n}(\alpha,\beta,c)$ such that for arguments $0<\beta\le\alpha<\varepsilon_0$ and $c\in \codes^n_\alpha$:
\begin{enumerate}
\item $\cmpscr{Rstr}_{n}(\alpha,\beta,c)$ is an element of $\codes^n_\beta$;
\item $\ev^n_\alpha(\cmpscr{Rstr}_{n}(\alpha,\beta,c))=\ev^n_\alpha(c)\cap U^n_\beta$;
\item $\ordcmp^n_\beta(\cmpscr{Rstr}_{n}(\alpha,\beta,c))\le \ordcmp^n_\alpha(c)$, $\width^n_\beta(\cmpscr{Rstr}_{n}(\alpha,\beta,c))\le \width^n_\alpha(c)$;
\item $\cmpscr{Rstr}_{n}$ running time is $\Olarge(\max(\ordcmp^n_\alpha(c),\cmp(\alpha),\cmp(\beta))\cdot(\width^n_\alpha(c))^n)$. 
\end{enumerate}
\end{lemma}
\begin{proof}We prove the Lemma by induction on $n$. The case of $n=0$ is trivial:
$$\cmpscr{Rstr}_{0}\colon (\alpha,\beta,c)\longmapsto c.$$ 
Suppose $n>0$. Consider an input $(\alpha,\beta,c)$, where $c=(m,\overline{A},\overline{d})\in\codes^n_\alpha$. We choose maximal $k<m$ such that the left end of $A_k$ is less than $\beta$; obviously, at least one such $k$ exists. Suppose $A_k=[\gamma,\delta)$. We are going to define vectors $\overline{B}$ and $\overline{e}$ consists of $B_0,\ldots,B_{k}$ and $e_0,\ldots,e_{k}$ correspondingly. We put $B_i=A_i$ and $e_i=d_i$ for $i<k$. We put $B_k=[\gamma,\beta)$ and $e_k=\cmpscr{Rstr}_n(\Ordlog(A_k),\Ordlog([\gamma,\beta)),c)$. Finally, we put $\cmpscr{Rstr}_{n}(\alpha,\beta,c)=(k+1,\overline{B},\overline{e})$. Straightforward check shows that this $\cmpscr{Rstr}_{n}$ enjoys all required conditions.
\end{proof}

\begin{lemma} \label{intersection_lemma} Suppose $n\ge 0$. Then there exists a computable function $\cmpscr{Intr}_{n}(\alpha,c_1,c_2)$ such that for arguments $0<\alpha<\varepsilon_0$ and  $c_1,c_2\in \codes^n_\alpha$:
\begin{enumerate}
\item $\cmpscr{Intr}_{n}(\alpha,c_1,c_2)$ is an element of $\codes^n_\alpha$;
\item $\ev^n_\alpha(\cmpscr{Intr}_{n}(\alpha,c_1,c_2))=\ev^n_\alpha(c_1)\cap \ev^n_\alpha(c_2)$;
\item $\ordcmp^n_\alpha(\cmpscr{Intr}_{n}(\alpha,c_1,c_2))\le \ordcmp^n_\alpha(c_1)+\ordcmp^n_\alpha(c_2)$;
\item $\width^n_\alpha(\cmpscr{Intr}_{n}(\alpha,c_1,c_2))\le \width^n_\alpha(c_1)+\width^n_\alpha(c_2)$;
\item $\cmpscr{Intr}_{n}(\alpha,c_1,c_2)$ running time is $\Olarge(\max(\ordcmp^n_\alpha(c_1),\ordcmp^n_\alpha(c_2),\cmp(\alpha)) \\ (\max(\width^n_\alpha(c_1),\width^n_\alpha(c_2)))^{n+1})$. 
\end{enumerate}
\end{lemma}
\begin{proof} We prove this Lemma by induction on $n$. The case  $n=0$ is trivial. Now assume that $n>0$. Below we describe  $\cmpscr{Intr}_n$. Suppose an input $(\alpha,c_1,c_2)$ is given, code $c_1=(m_1,\overline{A^{(1)}},\overline{d^{(1)}})\in \codes^n_\alpha$, and code $c_2=(m_2,\overline{A^{(2)}},\overline{d^{(2)}})\in \codes^n_\alpha$ are given. We consider all pairwise intersections $A_i\cap B_j$, where $i<m_1$ and $j<m_2$; all intersections of this form are $\emptyset$ or $[\beta,\gamma)$ for some ordinals $\beta$ and $\gamma$, $\beta<\gamma$. We choose all non-empty intersections of considered form and then sort them in the order induced by the ordinal comparison of left ends of intervals; by this procedure we obtain sequence $B_0,B_1,\ldots,B_{k-1}$. Thus we obtain the vector $\overline{B}$ of the length $k$. Obviously, $\bigsqcup\limits_{i<k} B_i=[0,\alpha)$  and for all $i<k-1$, the right end of $B_i$ is equals to the left end of $B_{i+1}$. It is easy to see that $k\le m_1+m_2$. We are going to define vectors $\overline{e^{(1)}}$, $\overline{e^{(2)}}$, and $\overline{e^{(3)}}$ of the length $k$. For every $i<k$, we find a unique $j<m_1$ such that $B_i\subset A^{(1)}_j$ and then put $e^{(1)}_i=\cmpscr{Rstr}_{n-1}(\Ordlog(A^{(1)}_j),\Ordlog(B_i),d^{(1)}_j)$. Similarly, for every $i<k$, we find a unique $j<m_2$ such that $B_i\subset A^{(2)}_j$ and then put $e^{(2)}_i=\cmpscr{Rstr}_{n-1}(\Ordlog(A^{(2)}_j),\Ordlog(B_i),d^{(2)}_j)$. For every $i<k$, we put $e^{(3)}_i=\cmpscr{Intr}_{n-1}(\Ordlog(B_i),e^{(1)},e^{(2)})$. And finally, we put $\cmpscr{Intr}_{n}(\alpha,c_1,c_2)=(k,\overline{B},\overline{e^{(3)}})$. With the use of inductive hypothesis it is easy to check that all conditions on $\cmpscr{Intr}_n$ holds; note that we use quadratic upper bound for the number of comparison operations in the sort algorithm in order to give upper bound for running time.
\end{proof}

We define $|\varphi|$ for all polymodal formulas $\varphi$: 
\begin{itemize}
\item $|\top|=|\bot|=|x|=1$;
\item $|\varphi\land \psi|=|\varphi \lor \psi|=|\varphi\to\psi|=|\varphi|+|\psi|+1$;
\item $|\diamn{n}\varphi|=|\varphi|+1$.
\end{itemize}

Now we will prove the Theorem \ref{n_fragment}.

\begin{proof} We consider some formula $\varphi\in L^n_0$. Obviously, within time $\Olarge(|\varphi|)$  we can find a formula $\varphi'\in L^n_0$ such that $$\GLP\vdash \varphi\;\longleftrightarrow\;\varphi',$$  the only connectives that are used in $\varphi'$ are $\land,\lnot,\diamn{0},\ldots,\diamn{n-1}$, and $|\varphi'|=\Olarge(|\varphi|)$. For every subformula $\psi$ of $\varphi'$ we will find a code $c_{\psi}\in\codes^n_{\omega_n}$ such that $\ev^n_{\omega_n}(c_{\psi})=\{w\in U^n_{\omega_n}\mid \mathcal{U}^n_{\omega_n},w\Vdash \psi\}$. We consider subformulas of $\varphi'$ in an order such  that every subformula $\psi$ is considered after all strict subformulas of $\psi$. If $\psi$ is $\bot$ then $c_{\psi}=\cmpscr{EmpS}_n(\omega_n)$.  If $\psi$ is $\lnot\psi'$ for some $\psi'$  then $c_{\psi}=\cmpscr{Cmpl}_n(\omega_n,c_{\psi'})$. If $\psi$ is $\diamn{k}\psi'$ for some $\psi'$ then $c_{\psi}=\cmpscr{RInv}_{n,k}(\omega_n,c_{\psi'})$. If $\psi$ is $\psi'\land \psi''$ for some $\psi'$ and some $\psi''$ then $c_{\psi}=\cmpscr{Intr}_{n}(\omega_n,c_{\psi'},c_{\psi''})$. We easily show by induction on length show that for every subformula $\psi$, we have
$$\width^n_{\omega_n}(c_{\psi})\le |\psi|\mbox{ and}$$
$$\ordcmp^n_{\omega_n}(c_{\psi})\le |\psi|\cdot n.$$
Thus the calculation of $c_{\psi}$ with the use of codes for previously considered subformulas takes time $\Olarge(|\varphi'|\cdot |\varphi'|^{n+1})$. And complete process of calculation of $c_{\phi'}$ takes time $\Olarge(|\varphi'|^{n+3})$.

Obviously, $\GLP\vdash \varphi$ iff $\cmpscr{IsEmp}_n(\omega_n,\cmpscr{Cmpl}_n(\omega_n,c_{\varphi'}))=1$. This gives us a required polynomial time algorithm with running time $\Olarge(|\varphi|^{n+3})$.
\end{proof}

\begin{remark} With sharper formulation and proof of the Lemma \ref{intersection_lemma} we may obtain  $\cmpscr{Intr}_n$ with running time $\Olarge(\max(\ordcmp^n_\alpha(c_1),\ordcmp^n_\alpha(c_2),\cmp(\alpha))\cdot(\max(\width^n_\alpha(c_1),\width^n_\alpha(c_2)))^{n})$; we don't present a proof of the last fact here. This will give a decision algorithm for the language $\GLP^n_0$ with better running time upper bound $\Olarge(|\varphi|^{n+2})$.
\end{remark} 

\section{Conclusion and perspectives}
E.~Dashkov considered the strongly positive fragment of $\GLP $. There is also the positive fragment of $\GLP $.
 For every $\alpha,\beta\le \omega$ we denote by $P^\alpha_\beta$ the set of all formulas of the form
$\varphi\longleftrightarrow \psi,$
where $\varphi$ and $\psi$ are built from the logical constant $\top$, conjunction, disjunction, modalities $\diamn{i}$ for $i<\alpha$, modalities $\boxn{i}$ for $i<\alpha$ and propositional variables $x_j$ for $j<\beta$. Consider $\GLP $-provability problems for formulas from $P^\alpha_\beta$, where either $\alpha=\omega$ or $\beta>0$. We conjecture that all these problems are $\textsc{PSPACE}$-complete.
 
If $n$ is large enough then the proof of the Theorem \ref{n_fragment} gives us highly ineffective decision algorithms for $\GLP ^n_0$. It is unknown are there effective algorithms for these problems. We conjecture that there are no uniform $N$ such that for every $n$ there is a decision algorithm for $\GLP^n_0$ with running time $\Olarge(|\varphi|^N)$. But from Theorem \ref{completeness_theorem} it follows that our conjecture implies $\textsc{PTIME}\ne \textsc{PSPACE}$. Thus if this conjecture holds then this problem seems to be very hard to solve without use of complexity-theoretic assumptions.

\bibliography{main}

\begin{thebibliography}{10}

\bibitem{Bek05-2}
L.~Beklemishev.
\newblock Veblen hierarchy in the context of provability algebras.
\newblock In {\em Logic, Methodology and Philosophy of Science, Proceedings of
  the Twelfth International Congress.}, pages 65--78. Kings College
  Publications, 2005.

\bibitem{Bek05}
L.D. Beklemishev.
\newblock Reflection principles and provability algebras in formal arithmetic.
\newblock {\em Russian Mathematical Surveys}, 60:197--268, 2005.

\bibitem{BekJooVer05}
Lev~D. Beklemishev, Joost~J. Joosten, and Marco Vervoort.
\newblock A finitary treatment of the closed fragment of {J}aparidze's
  provability logic.
\newblock {\em J. Log. Comput.}, 15(4):447--463, 2005.

\bibitem{ChagRyb02}
Alexander~V. Chagrov and Mikhail~N. Rybakov.
\newblock How many variables does one need to prove {PSPACE}-hardness of modal
  logics.
\newblock In {\em Advances in Modal Logic}, pages 71--82, 2002.

\bibitem{CookRec73}
Stephen~A. Cook and Robert~A. Reckhow.
\newblock Time bounded random access machines.
\newblock {\em J. Comput. Syst. Sci.}, 7(4):354--375, August 1973.

\bibitem{Dash12}
E.~Dashkov.
\newblock On the positive fragment of the polymodal provability logic {GLP}.
\newblock {\em Mathematical Notes}, 91:318--333, 2012.

\bibitem{Ign93}
K.N. Ignatiev.
\newblock On strong provability predicates and the associated modal logics.
\newblock {\em The Journal of Symbolic Logic}, 58(1):249--290, 1993.

\bibitem{Jap86}
G.K. Japaridze.
\newblock The modal logical means of investigation of provability.
\newblock {T}hesis in {P}hilosophy, in {R}ussian, Moscow, 1986.

\bibitem{Lad77}
Richard~E. Ladner.
\newblock The computational complexity of provability in systems of modal
  propositional logic.
\newblock {\em SIAM J. Comput.}, 6(3):467--480, 1977.

\bibitem{Shap08}
Ilya Shapirovsky.
\newblock {PSPACE}-decidability of {J}aparidze's polymodal logic.
\newblock In {\em Advances in Modal Logic}, pages 289--304, 2008.

\bibitem{Svej03}
V{\'\i}z{\v e}slav {\v S}vejdar.
\newblock The decision problem of provability logic with only one atom.
\newblock {\em Archive for Mathematical Logic}, 42(8):763--768, 2003.

\end{thebibliography}
\bibliographystyle{plain}
\end{document}